\documentclass{amsart}

\usepackage{amsmath}
\usepackage{amssymb}
\usepackage{graphicx}
\usepackage{mathrsfs}
\usepackage{tikz-cd}
\usepackage{xcolor}
\usepackage{tikz}
\usepackage[utf8]{inputenc}
\usepackage[english]{babel}
\usepackage{amsthm}
\usepackage{comment}
\usepackage{enumerate}
\usepackage{enumitem}

\newcommand{\Rep}{\mathrm{Rep}}
\newtheorem{theorem}{Theorem}[section]
\newtheorem{proposition}[theorem]{Proposition}

\newtheorem{corollary}[theorem]{Corollary}
\newtheorem{lemma}[theorem]{Lemma}

\theoremstyle{definition}
\newtheorem{remark}[theorem]{Remark}
\newtheorem{definition}[theorem]{Definition}
\newtheorem{example}[theorem]{Example}

\renewcommand{\thetheorem}{\Alph{theorem}}  

\usetikzlibrary{%
  matrix,%
  calc,%
  arrows%
}

\graphicspath{ {Images/} }

\newcommand{\C}{\mathbb{C}}
\newcommand{\R}{\mathbb{R}}
\newcommand{\Z}{\mathbb{Z}}
\newcommand{\Q}{\mathbb{Q}}
\newcommand{\N}{\mathbb{N}}

\newcommand{\TT}{\mathbb{T}}
\newcommand{\ZZ}{\mathcal{Z}}

\DeclareMathOperator{\id}{id}

\DeclareMathOperator{\im}{Im}

\DeclareMathOperator{\Ext}{Ext}

\DeclareMathOperator{\Hom}{Hom}

\DeclareMathOperator{\diag}{diag}

\DeclareMathOperator{\Prim}{Prim}
\DeclareMathOperator{\ind}{ind}
\DeclareMathOperator{\Span}{span}
\DeclareMathOperator{\SU}{SU}
\DeclareMathOperator{\U}{U}
\DeclareMathOperator{\PU}{PU}
\DeclareMathOperator{\nuc}{nuc}
\DeclareMathOperator{\SO}{SO}
\newcommand{\mt}{\mathtt}
\newcommand{\mc}{\mathcal}
\newcommand{\dimnuc}{\dim_{\nuc}}

\usepackage{esint}
\newcommand{\Irr}{\mathrm{Irr}}

\usepackage[backref=page, colorlinks]{hyperref}

\title[Nuclear Dimension of Twisted Algebras of
Virtually Abelian Groups]{Nuclear Dimension of Twisted $C^*$-Algebras of
Virtually Abelian Groups}

\author{Forrest Glebe}
\address{FG: Department of Mathematics
University of Hawaii at Manoa
2565 McCarthy Mall (Keller Hall)
Honolulu, HI 96822, USA}
\email[]{forrest6@hawaii.edu\hspace{10mm}Website: https://sites.google.com/view/forrestglebe/home}

\author{Pradyut Karmakar}
\address{PK: Department of Mathematics and Statistics, Sam Houston State University, Huntsville TX 77341, USA}
\email[]{karmakar.pradyut@gmail.com\hspace{10mm}
Website:https://sites.google.com/view/pradyut-karmakar}

\author{Iason Moutzouris}
\address{IM: Department of Mathematics and Statistics, Sam Houston State University, Huntsville TX 77341, USA}
\email[]{ixm089@shsu.edu \hspace{10mm} Website:https://sites.google.com/view/iasonmoutzouris/
}

\begin{document}
\begin{abstract}
    Let $G$ be a finitely generated virtually abelian group and $[\sigma]\in H^2(G;\TT)$ such that $\sigma(x,y)$ is always a root of unity. We show that the nuclear dimension of the twisted group $C^{\ast}$-algebra $C^*(G,\sigma)$ is equal to the rank of a finite index abelian subgroup of $G$. We also show that $\dimnuc(C^*(\Z^r,\sigma))=r$ if and only if $\sigma$ is type I.

\end{abstract}
\maketitle
\begingroup
\renewcommand\thefootnote{}

 \par
\endgroup
\section{Introduction}
The notion of the nuclear dimension was introduced by Winter and Zacharias in \cite{WZ10} as a noncommutative analogue of the topological covering dimension, and has since played a central role in the Elliott classification program for simple $C^*$-algebras. Indeed, in the case of simple, separable and nuclear $C^*$-algebras, the nuclear dimension can be $0$ (if the $C^*$-algebra is an AF-algebra), $1$ (if it absorbs the Jiang-Su algebra $\ZZ$ tensorially, but is not AF) or $+\infty$ (otherwise) \cite{dim_nuc_simple,dim_nuc_simple_st_proj}. \par
In this article, we focus to the case of twisted group $C^*$-algebras. Eckhardt and Wu showed in \cite{EcWu2026} that twisted group $C^*$-algebras of virtually polycyclic groups have finite nuclear dimension. However, finding the exact value of the nuclear dimension of these $C^*$-algebras has been an extremely challenging problem. Regarding the untwisted case, it was recently shown in \cite{paper_untwisted_case} that the nuclear dimension of the group $C^{\ast}$-algebra of a finitely generated virtually abelian group is exactly equal to the Hirsch length of the ambient group. Recall that, for such a group, its Hirsch length is equal to the rank of any abelian subgroup of finite index. For the definition of the Hirsch length in a larger class of groups, we refer the reader to \cite{hillman_finite_Hirsch_length_old}. The main purpose of this paper is to generalize the above result for a certain type of twisted group $C^*$-algebras.

\begin{theorem} (Theorem \ref{main theorem})
Let $G$ be a finitely generated, virtually abelian discrete group, and $[\sigma]\in H^2(G;\TT)$ such that $\sigma(x,y)$ is always a root of unity. Then, $$\dimnuc(C^*(G,\sigma))=h(G)$$
where $h(G)$ is the rank of a finite index abelian subgroup of $G$.

\end{theorem}

A cohomology class as in Theorem A is called \emph{rational}. Otherwise, it is called \emph{irrational} (see also Definition \ref{rational cocycle definition}). In Section 3, we will show that for finitely generated, virtually abelian groups, rational cohomology classes are exactly the ones that are Type I.\footnote{We refer the reader to \cite{kleppner_type_I_cocycles} for information about Type I cohomology classes.} \par
Hence, our result gives an affirmative answer to \cite[Question 2.14]{BL24} in the case that the group is finitely generated. \par
The first step in our proof is to show that these $C^*$-algebras are subhomogeneous, which allows us to compute their nuclear dimension by using Winter's result in \cite{Win04}. Moreover, we observe that their irreducible representations correspond to a certain subset of the irreducible representations of a certain finitely generated, virtually abelian group that arises as a central extension of $\Z_n$ by the group (see Remark \ref{correspondence}). This allows us to utilize the Mackey Machine (\cite[Thm. 4.28]{KanTay13}), as well as ideas from the proof of the untwisted case. \par
On the other hand, observe that the irrational rotation algebras arise as twisted group $C^*$-algebras of $\Z^r$ with irrational cohomology classes (Example \ref{rotation algebras type of cocycle}). But these $C^*$-algebras have nuclear dimension 1 (\cite[Example 6.1]{WZ10}), so we cannot drop the rational cohomology class assumption from Theorem A. We actually suspect that the nuclear dimension of the twisted group $C^*$-algebras is equal to the Hirsch length of the group G, \emph{only} when the cohomology class is rational. We have managed to verify this on the case when $G=\Z^r$.

\begin{theorem} (Theorem \ref{irrational case on Z^r}, Corollary \ref{equivalene of rational and dimnuc=rank})
    $\dimnuc C^{\ast}(\Z^r, \sigma)=r$ if and only if $[\sigma]$ is rational.
\end{theorem}
It is known that these twisted group $C^*$-algebras are exactly the higher dimensional noncommutative tori (see Definition \ref{higher dimensional non-commutative torus} and Remark \ref{all twisted group c* algebras are nc tori}). Our proof utilizes results of Elliott, Li, and Phillips \cite{elliott2006moritaequivalencesmoothnoncommutative,Morita_equivalence_nc_tori, simple_nc_torus}, which imply that when the cohomology class is not rational, the non-commutative tori is strongly Morita equivalent to a $C^*$-algebra of the form $C(X)\otimes D$, where $D$ is a unital, simple, A$\TT$ $C^*$-algebra. Finally, to deduce the upper bound, we heavily rely on a result of Tikuisis and Winter (\cite{nuclear_dimension_commutative_tensor_Z}, see also Theorem \ref{upper bound for dim_nuc of commutative tensor classifiable}). \par
The paper is structured as follows: In Section 2, we give background information on finitely generated, virtually abelian groups, group cohomology, projective representations, twisted group $C^{\ast}$-algebras,  and nuclear dimension. This information will be needed in the following Sections. In Section 3, we introduce the notion of a rational cohomology class, and give various equivalent conditions to that notion in the case of finitely generated, virtually abelian groups. These characterizations will heavily be used in Section 4, where we prove our main result. Finally, in Section 5, we prove Theorem B.

\section*{Acknowledgements}
We would like to thank Rufus Willett for letting us know about the results in \cite{nuclear_dimension_commutative_tensor_Z}. We would also like to thank the referee for their useful comments.

\section{Preliminaries}
\renewcommand{\thetheorem}{\thesection.\arabic{theorem}}
\subsection{Irreducible representations and Subhomogeneous $C^*$-algebras} \par
A $C^*$-algebra is \emph{subhomogeneous} if there exists $M>0$ such that for every irreducible representation $\pi:A \to B(\mc{H})$, we have that $\dim(\pi(A))\leq M$. If $A$ embeds in $C(X,M_n(\C))$ for some compact Hausdorff topological space $X$, and some $n\in \N$, then $A$ is subhomogeneous (\cite[IV. 1.4.4]{Bla10}). \par
Let $A$ be a subhomogeneous $C^*$-algebra, and $n\in \N$. We use $\Rep_n(A)$ to define the space of $n$-dimensional representations $A\to M_n(\C)$. We endow $\Rep_n(A)$ with the point-norm topology. We use $\Irr_n(A)$ to define the subspace of $\Rep_n(A)$ (with the subspace topology) that consists of irreducible representations. Two irreducible representations $\pi \colon A\to B(\mc{H})$ and $\pi'\colon A\rightarrow B(\mc{H}')$ are \emph{equivalent} if there exists a unitary operator $U \colon \mc{H}\to \mc{H}'$ such that $U\pi(a)=\pi'(a)U$ for all $a\in A$. In this case, we write $\pi\simeq \pi'$. We use $\Prim(A)$ to identify $\Irr(A)/\simeq$, endowed with the Jacobson topology\footnote{In \cite{Dix77}, and elsewhere in the literature, the quotient $\Irr(A)/\simeq$ is defined as the \emph{spectrum} of $A$ and is denoted with $\widehat{A}$. On the other hand, $\Prim(A)$ is defined as the set of primitive ideals of $A$. Because every primitive ideal is a kernel of some irreducible representation, there is a well-defined surjective map from $\widehat{A}$ onto $\Prim(A)$ that sends the class of an irreducible representation to its kernel (\cite[3.1.5.]{Dix77}). The Jacobson topology is a topology on $\Prim(A)$ (for the definition of this topology we refer the reader to \cite[3.1.]{Dix77}). But, for a subhomogeneous $C^*$-algebra $A$, the above map $\widehat{A}\to \Prim(A)$ is a homeomorphism (see \cite[3.1.6 (p.71)]{Dix77} and \cite[Thm IV.15.7 (p.339)]{Bla10}). This allows us to identify $\Irr(A)/\simeq$ with $\Prim(A)$, and define the Jacobson topology on $\Irr(A)/\simeq$.}. The subspace of $\Prim(A)$ that consists of the $n$-dimensional representations is denoted by $\Prim_n(A)$. For more information, we refer the reader to \cite[Chapter 3]{Dix77}.
\subsection{Two Cohomology}

We will provide a brief introduction to the two cohomology, focusing on information that will be needed in future sections. For more details, we refer the reader to ~\cite{coho}. 
\begin{definition}
If $G$ is a group, and $A$ is an abelian group, we define a {\em 2-cocycle} to be a function $\sigma \colon G\times G \to A$ satisfying the following equation
$$\partial\sigma(g,h,k):=\sigma(g,h)-\sigma(g,hk)+\sigma(gh,k)-\sigma(h,k)=0\,.$$
A {\em 2-coboundary} is a function that can be written in the form
$$\sigma(g,h)=\gamma(g)-\gamma(gh)+\gamma(h)$$
for some function $\gamma \colon G\rightarrow A$. Every 2-coboundary is a 2-cocycle, and $H^2(G;A)$ is defined to be the group of 2-cocycles, modulo the subgroup of 2-coboundaries. The group operation is pointwise addition. A cocycle is said to be {\em normalized} if $\sigma(e,g)=\sigma(g,e)=0$ for all $g\in G$. It is well known that any cohomology class can be represented by a normalized cocycle, and henceforth we assume that all cocycles mentioned are normalized. To see why, note that by setting $\partial\sigma(g,e,e)$ and $\partial\sigma(e,e,g)$ equal to zero, one can see that $\sigma(g,e)=\sigma(e,g)=\sigma(e,e)$. Since constants are clearly coboundaries, we can subtract the constant function of value $\sigma(e,e)$ to get a normalized cocycle. We denote the group of 2-cocycles as $Z^2(G;A)$. The cohomology class of a cocycle $\sigma$ will be denoted $[\sigma]$.
\end{definition}

Group cohomology, $H^2(\bullet; A)$ is a contravariant functor, and $H^2(G;\bullet)$ is a covariant functor with morphisms defined as follows. If $\varphi \colon G\rightarrow H$ is a group homomorphism then $\varphi^*:H^2(H;A)\rightarrow H^2(G;A)$ is defined by $\varphi^*([\sigma])=[\sigma\circ(\varphi\times\varphi)]$. If $T \colon A\rightarrow B$ is a morphism of abelian groups, we define $T_* \colon H^2(G;A)\rightarrow H^2(G;B)$ by $T_*([\sigma])=[T\circ\sigma]$.

\begin{definition}
Define $C_k(G)$ to be formal linear combinations of elements of~$G^k$. We write a typical element of $C_2(G)$ as
$$\sum_{j=1}^Nx_j[a_j|b_j]$$
with $a_j,b_j\in G$ and $x_j\in\Z$. Define $\partial_2 \colon C_2(G)\rightarrow C_1(G)$ by the equation
$$\partial_2[a|b]=[a]-[ab]+[b]$$
and $\partial_3\colon C_3(G)\rightarrow C_2(G)$ by
$$\partial_3[a|b|c]=[a|b]-[a|bc]+[ab|c]-[b|c].$$
Then $H_2(G;\Z):=\ker(\partial_2)/\im(\partial_3)$.
\end{definition}

\begin{definition}\label{kronecker pairing}
The {\em Kronecker pairing} between 2-homology and 2-cohomology is a bilinear map from $H^2(G;A)\times H_2(G;\Z)$ to $A$ defined by the formula
$$\left\langle\sigma,\sum_{j=1}^Nx_n[a_j|b_j]\right\rangle:=\sum_{j=1}^Nx_j\sigma(a_j,b_j)$$
where $\sigma$ is a cocycle, and $\sum_{j=1}^Nx_n[a_j|b_j]$ is a cycle. The value does not depend on either choice of representative.
\end{definition}

This induces a map $\kappa\colon H^2(G;A)\rightarrow  \Hom(H_2(G;\Z),A)$. This map appears in the Universal Coefficient Theorem, stated below:

\begin{theorem}\cite[Theorem 53.1]{eat}\label{UCT AT} There is an exact sequence
$$\begin{tikzcd}
0\arrow{r}{}&\Ext(H_1(G;\Z),A)\arrow{r}{}& H^2(G;A)\arrow{r}{\kappa}&\Hom(H_2(G;\Z),A)\arrow{r}{}&0
\end{tikzcd}.$$
The maps in this sequence are natural in the sense that they commute with maps induced by a group homomorphism between the groups $G$ and $H$, or a homomorphism of abelian groups from $A$ to $B$.
\end{theorem}

Since $H_1(G;\Z)$ is naturally isomorphic to the abelianization of $G$,~\cite[page 36]{coho}, we may take this as the definition.

\begin{definition}A {\em central extension} of $G$ by $A$ is a short exact sequence
$$\begin{tikzcd}
e\arrow{r} & A
\arrow{r}{} &
\widetilde{G} \arrow{r}{}
& G\arrow{r}{}& e
\end{tikzcd}$$
where the image of $A$ in $\widetilde{G}$ is central in $\widetilde{G}$. We say that two central extensions are {\em equivalent} if we can make a commutative diagram as follows:
$$\begin{tikzcd}
e\arrow{r} & A
\arrow{r}{}\arrow{d}{\id_A} &
\widetilde{G}_1\arrow{d}{\cong} \arrow{r}{}
& G\arrow{r}{}\arrow{d}{\id_G}& e
\\
e\arrow{r} & A
\arrow{r}{} &
\widetilde{G}_2 \arrow{r}{}
& G\arrow{r}{}& e.
\end{tikzcd}$$
\end{definition}

\begin{theorem}[{\cite[Theorem IV.3.12]{coho}}] As a set $H^2(G;A)$ is in bijection with the equivalence classes of central extensions of $G$ by $A$.
\end{theorem}

Given an explicit central extension, we may find a cocycle representative of the corresponding element of $H^2(G;A)$ as follows. Pick $\theta$ to be a set-theoretic section from $G$ to $\widetilde{G}$. Then viewing $A$ as a subset of $\widetilde{G}$ define
$$\sigma(g,h)=\theta(g)\theta(h)\theta(gh)^{-1}\in A\,.$$
By \cite[Equation IV.3.3]{coho}, this is a cocycle representative of the cohomology class corresponding to this central extension.\footnote{In order to guarantee that $\sigma$ is normalized we pick $\theta$ so that $\theta(e)=e$.} The zero cohomology class corresponds to the extension
$$\begin{tikzcd}
e\arrow{r} & A
\arrow{r}{} &
A\times G \arrow{r}{}
& G\arrow{r}{}& e
\end{tikzcd}$$
with the obvious maps.

\begin{proposition}\cite[Chapter IV.3 Exercise 1]{coho}\label{exercise}
Consider a 2-cohomology class $[\sigma]\in H^2(G;A)$ represented as a central extension:
$$\begin{tikzcd}
e\arrow{r} & A
\arrow{r}{} &
\widetilde{G} \arrow{r}{}
& G\arrow{r}{}& e.
\end{tikzcd}$$
If $\pi \colon H\rightarrow G$ is a group homomorphism and $T:A\rightarrow B$ is a homomorphism of abelian groups, then the pullback $\pi^*([\sigma])\in H^2(H;A)$ is given by the unique extension fitting into the top row of the diagram
$$\begin{tikzcd}
e\arrow{r} & A
\arrow{r}{}\arrow{d}{\id_A} &
\widetilde{H}\arrow{d} \arrow{r}{}
& H\arrow{r}{}\arrow{d}{\pi}& e
\\
e\arrow{r} & A
\arrow{r}{} &
\widetilde{G} \arrow{r}{}
& G\arrow{r}{}& e
\end{tikzcd}$$
and the push-forward $T_*([\sigma])\in H^2(G;B)$ is given by the unique extension fitting into the bottom row of the diagram:
$$\begin{tikzcd}
e\arrow{r} & A
\arrow{r}{}\arrow{d}{T} &
\widetilde{G}\arrow{d} \arrow{r}{}
& G\arrow{r}{}\arrow{d}{\id_G}& e
\\
e\arrow{r} & B
\arrow{r}{} &
\widetilde{G}' \arrow{r}{}
& G\arrow{r}{}& e.
\end{tikzcd}$$
\end{proposition}
We end the subsection by giving some background information regarding the projective and special unitary groups. This information will be useful in the proof of Proposition \ref{fd=>rational}.
\begin{remark}\label{PU(N)extension}
\begin{enumerate}
    \item Let $\mc{H}$ be a Hilbert space. The {\em projective unitary group}, is defined as $$\PU(\mc{H}):=\U(\mc{H})/\ZZ(\U(\mc{H}))\cong\U(\mc{H})/\TT.$$
    Recall that $$\ZZ(U(\mc{H}))=\{z\cdot 1_{\mc{H}}:z\in \TT\}\cong \TT. $$
\item Let $n\in \N$. The \emph{special unitary group} is defined as $$\SU(n):=\{A\in \U(n): \det(A)=1\}\,.$$ Notice that $$\ZZ(\SU(n))=\{\diag(\omega,...,\omega):\omega^n=1\}\cong \Z_n\,.$$
Let $\pi\colon\SU(n)\rightarrow \PU(n)$ defined as the composition of the natural embedding $\SU(n)\hookrightarrow \U(n)$ with the quotient map $\U(n)\rightarrow \PU(n)$. Notice that $\pi$ is surjective and $\ker(\pi)=\ZZ(\SU(n))\cong \Z_n$.

\item From all the above, we deduce the following commutative diagram
$$\begin{tikzcd}
e\arrow{r}&\Z_n\arrow{r}\arrow[d]&\SU(n)\arrow{d}\arrow{r}&\PU(n)\arrow{r}\arrow[d,"\id"]&e\\
e\arrow{r}&\TT\arrow{r}&\U(n)\arrow{r}&\PU(n)\arrow{r}&e
\end{tikzcd}$$
where the maps $\Z_n\hookrightarrow \TT$ and $\SU(n)\hookrightarrow \U(n)$ are the natural embeddings. \par
Notice that each short exact sequence is central.
\end{enumerate}
\end{remark} 
\subsection{Projective representations and twisted group $C^*$-algebras}

\begin{definition}
A {\em projective representation} of $G$ is a function $\varphi \colon G\rightarrow\U(\mc{H})$ such that for all $g,h\in G$, $\varphi(g)\varphi(h)\varphi(gh)^{-1}\in\TT1_{\mc{H}}$. Two projective representations $\varphi$, $\psi$ are {\em equivalent} if there is some function $\alpha \colon G\rightarrow \TT$ such that for all $g\in G$, $\varphi(g)=\alpha(g)\psi(g)$. 
\end{definition}
For a survey on projective representations we refer the reader to~\cite{projectiveRep}. \par
To each projective representation, $\varphi$, we may associate a cocycle $\sigma \colon G\times G\rightarrow\TT$ by the equation $\sigma(g,h)=\varphi(g)\varphi(h)\varphi(gh)^{-1}\in\TT$.\footnote{If $\varphi(e)=1_{\mc{H}}$ then $\sigma$ is normalized.} It is easy to check that this satisfies the cocycle identity and that if two projective representations are equivalent, their associated cocycles are in the same cohomology class.

To each projective representation $\varphi$, we may associate a genuine group homomorphism from $\hat\varphi \colon G\rightarrow\PU(\mc{H})$ defined by taking the image of $\varphi$ in $\PU(\mc{H})$. Clearly, $\hat\varphi=\hat\psi$ if and only if $\varphi$ is equivalent to $\psi$.

The following is likely well-known to experts.

\begin{proposition}\label{cocycle_pullback}
Let $\varphi$ be a projective representation, and let $\sigma$ be the associated cocycle. Then the cohomology class of $\sigma$ is given by the class of the extension in the bottom row of the diagram in Remark~\ref{PU(N)extension}, pulled back by $\hat{\varphi}$.
\end{proposition}

\begin{proof}
Pick $\theta$ to be any set-theoretic section from $\PU(\mc{H})$ to $\U(\mc{H})$. Let $\psi=\theta\circ\hat\varphi$. Clearly, $\psi$ is a projective representation that is equivalent to $\varphi$, and the cocycle associated to $\psi$ is given by $\chi(g,h)=\theta(\hat\varphi(g))\theta(\hat\varphi(h))\theta(\hat{\varphi}(gh))^{-1}\in\TT$ for all $g, h \in G$. The cocycle associated with the extension in the bottom row of the diagram in Remark~\ref{PU(N)extension} is given by $\eta(g,h)=\theta(g)\theta(h)\theta(gh)^{-1}\in\TT$ for all $g, h \in G$, and one can see that $\chi$ is the pullback of $\eta$ by $\hat{\varphi}$. Since $\varphi$ and $\psi$ are equivalent, it follows that $\sigma$ is in the same cohomology class as $\eta$.
\end{proof}

\begin{definition}
Let $G$ be a discrete group and $\sigma\in Z^2(G;\TT)$. The {\em twisted group algebra with cocycle $\sigma$} denoted $\C[G,\sigma]$, is the algebra of finite linear combinations of elements of the group (denoted $\eta_g$ with $g\in G$) with multiplication given by
$$\eta_g\eta_h=\sigma(g,h)\eta_{gh} \quad g, h \in G.$$
There is an adjoint structure given by
$$\eta_g^*=\sigma(g,g^{-1})^{-1}\eta_{g^{-1}} \quad g \in G.$$
\end{definition}

It is easy to check that if $\sigma$ and $\chi$ are in the same cohomology class, then $\C[G,\sigma]\cong\C[G,\chi]$ and if $$\sigma(g,h)=\chi(g,h)\cdot\gamma(g)\cdot\gamma(gh)^{-1}\cdot\gamma(h),$$
then $\eta_g\mapsto\gamma(g)\eta_g$ is an isomorphism from $\C[G,\sigma]$ to $\C[G,\chi]$.

\begin{definition}
The {\em twisted full group $C^*$-algebra}, denoted $C^*(G,\sigma)$ is defined to be the completion of $\C[G,\sigma]$ with the norm given by the supremum of the induced norms given by any *-representation of $\C[G,\sigma]$. From the above, we deduce that if $\sigma$ and $\chi$ are in the same cohomology class, then $C^*(G,\sigma)\cong C^*(G,\chi)$.
\end{definition}
\begin{remark}
The representations of $C^*(G,\sigma)$ are in one-to-one correspondence with the projective representations of $G$ with cocycle $\sigma$. Indeed, by utilizing the above definitions, any projective representation of $G$ with cocycle $\sigma$ can be extended to a representation $\C[G,\sigma]$. Thus, we obtain a representation of $C^*(G,\sigma)$ by applying the universal property of the twisted full group $C^*$-algebra. Conversely, if $\pi \colon C^*(G,\sigma)\to B(\mathcal{H})$ is a representation, then its restriction to $G$ is a projective representation.
\end{remark}

\begin{remark}\label{reps_of_central_extension}
Suppose that $\sigma\in H^2(G;\TT)$ is a cocycle taking values in $n$th roots of unity. 
Then we may think of it as representing a cohomology class in $H^2(G;\Z_n)$ and view that class as an extension,
$$\begin{tikzcd}
e\arrow{r}&\Z_n\arrow{r}&\widetilde{G}\arrow{r}{q}&G\arrow{r}&e.
\end{tikzcd}$$
Let $\chi\colon \Z_n\rightarrow \TT$ be the character that sends the generator of $\Z_n$ to $\exp(\frac{2\pi i}{n})$.
Twisted representations of $G$ with cocycle $\sigma$ correspond to representations of $\widetilde{G}$ whose restriction to $\Z_n$ is $\chi$. Indeed, $\sigma$ corresponds to a set-theoretic section $\theta$ from $G$ to $\widetilde{G}$. By this correspondence, we see that $\chi(\theta(g)\theta(h)\theta(gh)^{-1})=\sigma(g,h)$.
Consequently, for any representation $\pi$ of $\widetilde{G}$ that restricts to $\chi$, the composition $\pi\circ\theta$ is a projective representation with cocycle $\sigma$. Conversely, given a projective representation $\varphi$ of $G$ with cocycle $\sigma$, one can check that
$g\mapsto \varphi(q(g))\chi(g\cdot\theta(q(g))^{-1})$ is a representation of $\widetilde{G}$ that restricts to $\chi$. 
\end{remark}
\begin{remark}\label{correspondence}

From the above two preceding remarks, we conclude that there exists $\omega\in \TT$ such that $\omega^n=1$ such that representations of $C^*(G,\sigma)$ correspond to representations of $C^*(\widetilde{G})$, where the generator of $\Z_n$ (viewed as a subgroup of $\widetilde{G}$), which we will call $a$, is sent to $\omega\cdot1$. Moreover, we can define a surjective $*$-homomorphism from $C^*(\widetilde{G})$ to $C^*(G,\sigma)$. Indeed, we can define a mutliplicative map $\varphi\colon\C[\widetilde{G}]\to C^*(G,\sigma)$ via the formula in Remark~\ref{reps_of_central_extension}. The Universal property of $C^*(\widetilde{G})$ allows us to extend to a surjective $*$-homorohpism $\Phi\colon C^*(\widetilde{G})\rightarrow C^*(G,\sigma)$. By~\cite[Theorem 5.5.7]{murph} it follows that the induced inclusion from $\Prim(C^*(G,\sigma))$ to $\Prim(C^*(\widetilde{G}))$ is an embedding. Hence, $\Prim(C^*(G,\sigma))$ is homeomorphic to the subspace of $C^*(\widetilde{G})$ that consists of the representations that send $a$ to $\omega\cdot1$.
\end{remark}

\subsection{Finitely generated, virtually abelian groups} 
A group $G$ is \emph{virtually abelian} if there exists an abelian subgroup $H$ of finite index. In the case $G$ is finitely generated, then so is $H$. Hence, $H$ contains a torsion-free subgroup of finite index, say $K\cong\Z^r$. By taking normal core of $K$ in $G$, there exists $N\unlhd G$ such that $[G:N]<\infty$ and $N\leq K$. Because $N$ has finite index in $K\cong \Z^r$, it follows that $N\cong \Z^r$. We gather the above observations into the following remark.
\begin{remark}\label{virtually_abelian_general_form}
    A group $G$ is finitely generated and virtually abelian if and only if $G$ fits into a short exact sequence of the form
    \begin{equation}\label{eqn:virtually_abelian_defn}
         e\rightarrow \Z^r\stackrel{i}{\rightarrow} G\stackrel{s}{\rightarrow}D\rightarrow e
    \end{equation}
    with $|D|<\infty$. 
\end{remark}

The number $r$ above is the {\em rank of $G$}. It is also called the \emph{Hirsch length} (we write $h(G)=r$). It coincides with the more general definition of Hirsch length found in~\cite{hillman_finite_Hirsch_length_old}. \par
Let $G$ be virtually abelian and identify $i(\Z^r)\unlhd  G$ with $\Z^r$ where we treat $\Z^r$ as a multiplicative group. Because $\Z^r$ is normal in $G$, there is a natural action of $G$ on $\Z^r$ defined by $g\cdot a=gag^{-1}$ for all $g\in G$, $a\in \Z^r$. Let $\gamma \colon D\rightarrow G$ be a section with $\gamma(1_D)=1_G$. Then, the action of $G$ on $\Z^r$ $(G\curvearrowright \Z^r$) descends to an action of $D$ on $\Z^r$ by $d\cdot a=\gamma(d)\cdot a$. Notice the induced action is independent of the section we choose. 

We also have an induced (left) action $G\curvearrowright\widehat{\Z}^r$ given by 
\[(g\cdot \chi)(a)=\chi(g^{-1}ag)\text{ for all }g\in G, \chi\in\widehat{\Z}^r, a\in \Z^r.\]
This action descends to an action of $D$ on $\widehat{\Z}^r$. For each $\chi\in \widehat{\Z}^r$, we define
\begin{equation}\label{eqn:stabilizer_and_orbit}
G_{\chi}=\{g\in G\,\colon\,g\cdot \chi=\chi\}\quad\text{ and }\quad\mc{O}_{\chi}=\{g\cdot\chi\,\colon\,g\in G\}
\end{equation}
to be the stabilizer subgroup associated to $\chi$ and the orbit associated to $\chi$, respectively. We observe that $|\mc{O}_{\chi}|={|G/ G_{\chi}|}$, $\Z^r\leq G_{\chi}$, and $|\mc{O}_{\chi}|$ divides $|D|$ for all $\chi\in\widehat{\Z}^r$. \par
The following Lemma is probably known to experts, but we include its proof for the sake of completion. It will be needed in Section 3.
\begin{lemma}\label{extVA}
Suppose that $G$ is a finitely generated virtually abelian group, and the following is a central extension
$$\begin{tikzcd}
e\arrow{r}&\Z_n\arrow{r}&\widetilde{G}\arrow{r}&G\arrow{r}&e.
\end{tikzcd}$$
Then $\widetilde{G}$ is also virtually abelian. Moreover, the cohomology class of the extension is trivial when restricted to a finite-index free abelian subgroup.
\end{lemma}

\begin{proof}
Since $G$ is virtually abelian and finitely generated, Remark \ref{virtually_abelian_general_form} implies that it has a finite-index subgroup isomorphic to $\Z^r$. Let $\iota$ be as in (\ref{eqn:virtually_abelian_defn}). We have the following commutative diagram:
$$\begin{tikzcd}
e\arrow{r}&\Z_n\arrow{r}\arrow[d, "\id"]&\widetilde{G}'\arrow{r}\arrow{d}&\Z^r\arrow{r}\arrow[d,"\iota"]&e\\
e\arrow{r}&\Z_n\arrow{r}&\widetilde{G}\arrow{r}&G\arrow{r}&e.
\end{tikzcd}$$
By the K\"unneth Theorem~\cite[Theorem 3.16]{hatcher}, each element of $H^2(\Z^r;\Z_n)$ can be written as a linear combination of cup products of 1-cohomology classes. Since $H^1(\Z^r; \Z_n)$ is naturally isomorphic to $\Hom(\Z^r,\Z_n)$, we see that the map from $\Z^r$ to itself induced by multiplication by $n$ induces the zero map on the level of 1-cohomology, and thus induces the zero map on the level of 2-cohomology as well. It follows that we have the following commutative diagram
$$\begin{tikzcd}
e\arrow{r}&\Z_n\arrow{r}\arrow[d,"\id"]&\Z_n\oplus\Z^r\arrow{r}\arrow{d}&\Z^r\arrow{r}\arrow{d}{\bullet n}&e\\
e\arrow{r}&\Z_n\arrow{r}\arrow[d,"\id"]&\widetilde{G}'\arrow{r}\arrow{d}&\Z^r\arrow{r}\arrow[d,"\iota"]&e\\
e\arrow{r}&\Z_n\arrow{r}&\widetilde{G}\arrow{r}&G\arrow{r}&e.
\end{tikzcd}$$
Notice that we have the following list of finite-index inclusions:
$$\Z^r\hookrightarrow \Z_n\oplus\Z^r\hookrightarrow \widetilde{G}'\hookrightarrow \widetilde{G}\,.$$
It follows that $\widetilde{G}$ is a finitely generated, virtually abelian group. Finally, notice that the cohomology class of the extension, when restricted to the finite-index subgroup $\Z^r\leq G$ is trivial because the top row is a direct sum extension.
\end{proof}
We end the subsection by gathering known results about group $C^*$-algebras of finitely generated, virtually abelian groups.
\begin{remark}\label{remarks about C*-algebras of fg VA groups}
    Let $G$ be a finitely generated, virtually abelian group.
    \begin{enumerate}
        \item For $D$ and $r$ as in (\ref{eqn:virtually_abelian_defn}), $C^*(G)$ embeds into $M_{|D|}(C(\TT^r))$ (see \cite[Section 3.2]{Connective} for the construction of the embedding). Hence $C^*(G)$ is subhomogeneous.
        \item $\Prim(C^*(G))$ is not Hausdorff in general. An explicit example is the Dimension 3 Crystallographic group 90. For an explanation of why this occurs, we refer the reader to \cite[Proposition 4.5, Section 5 and Appendix]{CW24}. However, for each $k\in \N$, $\Prim_k(C^*(G))$ is totally normal, and thus Hausdorff (\cite[Proposition 4.10]{paper_untwisted_case}).
    \end{enumerate}
\end{remark}
\subsection{Nuclear Dimension}
The notion of the {\em nuclear dimension} was introduced by Winter and Zacharias in \cite{WZ10}. In that paper, they showed that $\dimnuc (C(X))=\dim(X)$ for every locally compact second countable Hausdorff space $X$. In this sense, nuclear dimension can be viewed as a non-commutative analogue of the covering dimension. 

We refer the reader to \cite{WZ10} for the precise definition and basic properties of nuclear dimension. \par
In this paper, we are interested in computing the nuclear dimension on the setting of subhomogeneous $C^*$-algebras. For such $C^*$-algebras, Winter has shown that it is connected with the dimensions of the spaces of $k$-dimensional irreducible representations.
\begin{theorem}[cf. Main Theorem, \cite{Win04}]\label{dim_nuc_subhomogeneous}
 Let $A$ be a separable subhomogeneous $C^*$-algebra. Then 
\[\dimnuc (A)=\max_{i\in \N}\{\dim \Prim_i(A)\}.\]
Here $\dim$ denotes the covering dimension.
 \end{theorem}
We remark that the statement of the Main Theorem in \cite{Win04} is slightly different than presented here. For the exact statement, see \cite[Thm. 2.7]{paper_untwisted_case}). \par
It is already known that $\dimnuc (C^*(G, \sigma))\leq h(G)$ when $G$ is a finitely generated virtually abelian group and $[\sigma]$ is type I \cite[Prop. 2.14]{BL24}.\footnote{ The result is stated in terms of the asymptotic dimension, $\text{asdim}\,(G)$. However, $\text{asdim}\,(G)=h(G)$ for every finitely generated virtually abelian group $G$ by \cite[Thm. 3.5]{asymp_dim_va_groups}. Moreover, we refer the reader to Definition \ref{type I cocycle definition} for the definition of type I cohomology class.} Our main result will show that equality holds, and thus gives an affirmative answer to \cite[Question 2.15]{BL24}.
\section{A characterization of rational cohomology classes on virtually abelian groups}
\begin{definition}\label{rational cocycle definition}
A cohomology class $[\sigma]\in H^2(G;\TT)$ is {\em rational} if it has a representative $\sigma \colon G\times G\to \TT$ such that, for every $x,y\in G$, there exists a rational number $\theta=\theta(x,y)$, such that $\sigma(x,y)=e^{2\pi i\theta(x,y)}$. Otherwise, we say that $[\sigma]$ is \emph{irrational}.
\end{definition}
\begin{example}\label{rotation algebras type of cocycle}
    Let $G=\Z^2$, and $\theta\in (0,1)$. Define a 2-cocycle $\sigma$ via $\sigma \colon \Z^2\times \Z^2\to \TT$ given by $\sigma_{\theta}((x_1,x_2),(y_1,y_2))=e^{2\pi i\theta x_2y_1}$. Notice that $[\sigma_{\theta}]$ is rational if and only if $\theta\in \Q$. Moreover, it is known that $C^*(G,\sigma_{\theta})\cong A_{\theta}$, where $A_{\theta}$ is the rotation algebra. \par
    It is known that if $\theta=\frac{p}{q}\in \Q$, where gcd$(p,q)=1$, then $A_{\theta}\hookrightarrow M_q(C(S^2))$ (\cite[Theorem 1.2]{rational_rotation_algebras_paper}), and hence it is a subhomogeneous $C^*$-algebra. Moreover it has nuclear dimension 2 by \cite[Example 6.1]{WZ10}. \par
    On the other hand, if $\theta\notin \Q$, then $A_{\theta}$ is a simple, $A\TT$-algebra by \cite[Theorem 4]{irrational_rotation_algebras_are_AT}. Hence it is not subhomogeneous. Moreover it has nuclear dimension 1 by \cite[Example 6.1]{WZ10}.
\end{example}

\begin{lemma}\label{rational implies n-th root}
Let $G$ be a group such that $H_2(G;\Z)$ is finitely generated, and let $[\sigma]\in H^2(G;\TT)$ be rational. Then there exists some $n\in\N$ and some representative $\sigma'\in[\sigma]$ so that $\sigma'(x,y)$ is an $n$th root of unity for all $x,y\in G$.
\end{lemma}

\begin{proof}
     Note that we may abuse notation slightly to consider $\sigma$ to be a cocycle with values in $\Q/\Z$. The universal coefficient theorem (Theorem \ref{UCT AT}) yields the following short exact sequence

\begin{tikzcd}[column sep=small]
0 \arrow[r] &
\Ext(H_1(G; \mathbb Z);\mathbb Q/\mathbb Z) \arrow[r] &
H^2(G; \mathbb Q/\mathbb Z) \arrow[r,"\kappa"] &
\Hom(H_2(G; \mathbb Z); \mathbb Q/\mathbb Z) \arrow[r] &
0\,.
\end{tikzcd}

Since $\Q/\Z$ is a divisible group, $\Ext(H_1(G; \Z),\Q/\Z)=\{0\}$, so $\kappa$ is an isomorphism. Since $H_2(G;\Z)$ is finitely generated, it follows that the $\kappa([\sigma])(H_2(G,\Z))$ is in some finitely generated subgroup of $\Q/\Z$, which must be isomorphic to $\Z_n$ for some $n$. By the naturality of the sequence in the universal coefficient theorem, the inclusion $\Z_n\hookrightarrow \Q/\Z$ induces a commutative diagram:

\[
\begin{tikzcd}[column sep=2.2em, row sep=2em]
0 \arrow{r}
  & {\scriptsize \Ext(H_1(G;\Z),\Z_n)} \arrow{r} \arrow{d}
  & {\scriptsize H^2(G;\Z_n)} \arrow[r, "\widetilde{\kappa}"] \arrow[d,"\beta"]
  & {\scriptsize \Hom(H_2(G;\Z),\Z_n)} \arrow[d, "\gamma"] \arrow{r}
  & 0 \\
0 \arrow{r}
  & 0 \arrow{r}
  & {\scriptsize H^2(G;\Q/\Z)} \arrow{r}{\kappa}
  & {\scriptsize \Hom(H_2(G;\Z),\Q/\Z)} \arrow{r}
  & 0\,.
\end{tikzcd}
\]

Recall that the image of the homomorphism $\kappa([\sigma])$ is inside $\Z_n$. It follows that 
\begin{equation}\label{100}
 \kappa([\sigma])=\gamma([\tau])\text{ for some }\tau\in \Hom(H_2(G;\Z),\Z_n).   
\end{equation}
 Moreover, $\widetilde{\kappa}$ is surjective, and thus
 \begin{equation}\label{101}
  [\tau]=\widetilde{\kappa}([\sigma'])\,.   
 \end{equation}
 Commutativity of the diagram, together with (\ref{100}) and (\ref{101}), imply that $$\kappa([\sigma])=(\kappa \circ \beta)([\sigma'])\,.$$
 Finally, injective of $\kappa$ implies that $[\sigma]=\beta[\sigma']$.

\end{proof}

If $G$ is a finitely presented group, then $H_2(G;\Z)$ is finitely generated~\cite[page 197]{coho}. The assumption that $H_2(G;\Z)$ is finitely generated is necessary. The following example shows that
assuming instead that the group itself is finitely generated is not sufficient.

\begin{example}
Let $G=\Z\wr\Z$. We may represent any element of $G$ as a pair $(v,n)$ with $v\in G_0:=\bigoplus_{k=-\infty}^\infty\Z$ and $n\in\Z$. Letting $T$ be the bilateral shift on $G_0$, we may write the formula for multiplication on $G$ as
$$(v_1,n_1)\cdot(v_2,n_2)=(v_1+T^{n_1}v_2,n_1+n_2)\,.$$
Define $\sigma_0 \colon G_0\times G_0\rightarrow\Q$ by the formula
$$\sigma_0(v_1,v_2)=\sum_{k=1}^\infty\frac1k\langle v_1,T^kv_2\rangle\,.$$
This sum is actually finite for any fixed $v_1$ and $v_2$ because they have finite support. Clearly $\sigma_0$ is bilinear and also has the property that $\sigma_0(Tv_1,Tv_2)=\sigma_0(v_1,v_2)$. We define $\sigma \colon G\times G\rightarrow\Q$ as
$$\sigma((v_1,n_1),(v_2,n_2))=\sigma_0(v_1,T^{n_1}v_2)\,.$$
Then
\begin{align*}
\partial\sigma((v_1,n_1),(v_2,n_2),(v_3,n_3))=&\sigma_0(v_1,T^{n_1}v_2)-\sigma_0(v_1,T^{n_1}v_2+T^{n_1+n_2}v_3)\\
&+\sigma_0(v_1+T^{n_1}v_2,T^{n_1+n_2}v_3)-\sigma_0(v_2,T^{n_2}v_3)\\
=&\sigma_0(T^{n_1}v_2,T^{n_1+n_2}v_3)-\sigma_0(v_2,T^{n_2}v_3)\\
=&0\,.
\end{align*}
Set $\alpha:=\exp(2\pi i\sigma)$, and notice that $\alpha$ is 2-cocycle on $G$ that takes values in roots of unity. Let $\{e_n,n\in \Z\}$ be the canonical basis of $G_0$. For every $\ell,j\in \Z$ with $\ell>j$ we have
\begin{equation}\label{12}
    \sigma_0(e_{\ell},e_j)=\sum_{k=1}^\infty\frac1k\langle e_{\ell},T^ke_j\rangle=\frac{1}{\ell-j}\langle e_{\ell},e_{\ell} \rangle=\frac{1}{\ell-j}  \hspace{6mm} \text{and}
\end{equation}
\begin{equation}\label{13}
\sigma_0(e_j,e_{\ell})=\sum_{k=1}^\infty\frac1k\langle e_j,T^ke_{\ell}\rangle=0\,.    
\end{equation}
Hence, by the definition of the Kronecker pairing (Definition \ref{kronecker pairing}), (\ref{12}), and (\ref{13}), we deduce that
\begin{equation}\label{14}
  \left\langle\alpha,[e_{\ell}|e_j]-[e_j|e_{\ell}]\right\rangle=\exp(2\pi i(\sigma_0(e_{\ell},e_j)-\sigma_0(e_j,e_{\ell})=\exp(\frac{2\pi i}{\ell-j}) \,.
\end{equation}
Hence $\kappa([\alpha])$ has order at least $\ell-j$. Recall that $\kappa\colon H^2(G;\TT)\rightarrow \Hom(H_2(G;\Z),\TT)$ is the map that shows on the Universal Coefficient Theorem (Theorem \ref{UCT AT}). Thus $[\alpha]\in H^2(G,\TT)$ has order at least $\ell-j$. But $\ell$ and $j$ are arbitrary (with the only restriction that $\ell>j$). Hence $[\alpha]$ has infinite order, and thus there is no $n\in \N$ such that $\alpha(x,y)$ is an $n$th root of unity for every $x,y\in G$.

\end{example}

The following definition actually follows from a characterization of a type I $C^*$-algebra, due to Holzner~\cite{Ho1981} and independently  Kleppner (\cite{kleppner_type_I_cocycles}), that is applied to the case when the group is finitely generated, virtually abelian.\footnote{Note that these authors use a different definition of ``normalized'' than we do. See~\cite{KleppnerOlder} for an explanation of their notation.}
\begin{definition}[Theorem 2, \cite{kleppner_type_I_cocycles}]\label{type I cocycle definition} Let $G$ be a finitely generated, virtually abelian group, and $[\sigma]\in H^2(G;\Z)$. We say that $[\sigma]$ is \emph{type I} if there exists an (abelian) subgroup $N$ of $G$ with finite index such that $[\sigma]$ is trivial in $N$.
    
\end{definition}
Notice that by the proof of \cite[Prop. 2.14]{BL24} we can take $N$ to be normal in $G$. \par
We need the following Lemma.

\begin{lemma}\label{type I implies torsion}
Let $G$ be a discrete group and $N\le G$ be of finite index. If $[\sigma]\in H^2(G;A)$ and the restriction of $[\sigma]$ to $N$ is trivial, then $[\sigma]$ is $[G:N]$-torsion.
\end{lemma}
\begin{proof}
Let $\mathrm{res}_N^G:H^2(G;A)\rightarrow H^2(N;A)$ denote the restriction map. By \cite[III Proposition 9.5 (ii)]{coho} there is a map $\mathrm{cor}_N^G:H^2(N;A)\rightarrow H^2(G;A)$ so that $\mathrm{cor}_N^G\circ\mathrm{res}_N^G$ is multiplication by $[G:N]$. Thus $[G:M]\cdot[\sigma]=\mathrm{cor}_N^G\circ\mathrm{res}_N^G([\sigma])=0$.
\end{proof}

\begin{lemma}\label{type I implies root of unity}
Let $G$ be a discrete group and $\sigma$ be $n$-torsion.
Then the cohomology class $[\sigma] \in H^2(G;\TT)$ is represented by a
$2$-cocycle whose values are roots of unity. In particular, $[\sigma]$ is rational.
\end{lemma}
\begin{proof}
Let $\mu_n$ be the group of $n$th roots of unity in $\TT$.
Consider the short exact sequence of coefficients $e \to \mu_n \to \TT \xrightarrow{\zeta \mapsto \zeta^n} \TT \to e$. This induces a long exact sequence in cohomology:
$$
\begin{tikzcd}
\cdots \arrow[r] & {H^2(G; \mu_n)} \arrow[r, "i_{\ast}"] & {H^2(G;\TT)} \arrow[r, "\cdot n"] & {H^2(G;\TT )}\arrow[r] &\cdots
\end{tikzcd}
$$
by \cite[Section 3.E]{hatcher}.

Since $n[\sigma] = 0$, the class $[\sigma]$ lies in $\ker(\cdot n)$, which by exactness is $\mathrm{im}(i_*)$. Thus, $[\sigma]$ can be represented by a 2-cocycle $\sigma'$ taking values in the $n$-th roots of unity $\mu_n \subset \TT$. The result holds as $[\sigma] = [\sigma']$ in $H^2(G; \TT)$.
\end{proof}

It turns out that finite-dimensional representations exist only when the cohomology class is rational. This is probably well known, but we provide a proof below.

\begin{proposition}\label{fd=>rational}
Let $G$ be a discrete group. Suppose that $\rho \colon G\rightarrow \U(n)$ is a projective representation with associated cohomology class $[\omega]\in H^2(G;\TT)$. Then $[\omega]$ can be represented by a cocycle that takes values in the $n$th roots of unity.
\end{proposition}

\begin{proof}
Let $\widehat{\rho}\colon G\rightarrow \PU(n)$ be the associated genuine homomorphism. Then by Proposition~\ref{cocycle_pullback}, $[\omega]=\widehat{\rho}^*(\alpha)$ where $\alpha$ is the cohomology class associated to the central extension:
$$\begin{tikzcd}
e\arrow{r}&\TT\arrow{r}&\U(n)\arrow{r}&\PU(n)\arrow{r}&e.
\end{tikzcd}$$
From Remark \ref{PU(N)extension}, we have the following commutative diagram.
$$\begin{tikzcd}
e\arrow{r}&\Z_n\arrow{r}\arrow[d,"\iota"]&\SU(n)\arrow{d}\arrow{r}&\PU(n)\arrow{r}\arrow{d}&e\\
e\arrow{r}&\TT\arrow{r}&\U(n)\arrow{r}&\PU(n)\arrow{r}&e.
\end{tikzcd}$$
From Proposition \ref{exercise}, we deduce that $[\alpha]=\iota_*([\eta])$, where $\eta$ is the 2-cocycle corresponding to the top row of the diagram. \par
Hence, as a push-forward of a cohomology class with coefficients in $\Z_n$, it follows that $[\alpha]$ can be expressed by a cocycle with representatives in $n$th roots of unity. It follows that $[\omega]$ can be represented by a cocycle that takes values in the $n$th roots of unity.
\end{proof}

\begin{remark}
The converse of Proposition~\ref{fd=>rational} is not true for all groups. A class of counterexamples comes from quotients of so-called {\em Deligne-type groups}. These are finitely generated groups with a central 2-torsion element $d$ in the intersection of all finite-index subgroups. For examples of Deligne-type groups see~\cite[section 5]{Deligne,hillnonrf, RFstover,nonaproximable}. Let $G$ be a Deligne-type group $G$, and $d$ as above. Let also $\widetilde{G}=G/\langle d\rangle$, and $[\sigma]\in H^2(\widetilde{G}, \Z_2)$ be the cohomology class associated to the central extension 
$$\begin{tikzcd}
e\arrow{r}&\Z_2\arrow{r}&G\arrow{r}&\widetilde{G}\arrow{r}&e.
\end{tikzcd}$$
Let $\iota:\Z_2\hookrightarrow \TT$ the natural embedding and $[\widetilde{\sigma}]=\iota_*([\sigma])\in H^2(\widetilde{G},\TT)$ the push-forward. Notice that there is a representative $\widetilde{\sigma}$ that takes values in $\{1,-1\}\subset \TT$. We will show that $C^*(\widetilde{G},\widetilde{\sigma})$ does not admit any finite-dimensional representation. Assume, for the sake of contraction, that it has one. Then, by Remark~\ref{correspondence}, there exists a finite-dimensional representation $\pi:G\to U(k)$ such that $\pi(d)=-1$. Notice that because $G$ is finitely generated, $\pi(G)$ is a finitely generated linear group, and thus residually finite by Mal'cev's theorem (see for example \cite[Theorem 6.4.13]{brownozawa}). Hence, there exists a finite group $F$ and a group homomorphism $\rho:\pi(G)\to F$ such that $\rho(-1)\neq e$. Thus $(\rho\circ \pi)(d)\neq e$. However, this contradicts our assumption that $d$ is in the intersection of all finite-index subgroups of $G$.

\end{remark}

We combine the above on the following Theorem.
\begin{theorem}\label{characterization of rational cocycle}

    Let $G$ be a finitely generated, virtually abelian group and $[\sigma]\in H^2(G;\TT)$. The following are equivalent.
    \begin{enumerate}[label=\roman*.]
        \item $[\sigma]$ is rational.
        \item There exists $n\in \N$ and a representative $\sigma$ which takes values in $n$-th roots of unity.
        \item $[\sigma]$ is type I.
        \item $[\sigma]$ is a torsion element in $H^2(G; \TT)$.
        \item $C^*(G,\sigma)$ is subhomogeneous.
        \item $C^*(G,\sigma)$ admits a finite-dimensional representation.
    \end{enumerate}
    \begin{proof}
    $(i)\Leftrightarrow(ii)$ Recall that finitely generated, virtually abelian groups are finitely presented. Hence $H_2(G;\Z)$ is finitely generated (see discussion after Lemma \ref{rational implies n-th root}). So, one direction follows from Lemma \ref{rational implies n-th root}. The other direction is immediate. \par
    $(ii)\Rightarrow (iii)$ It follows from Lemma \ref{extVA}. \par
    \par
    $(iii)\Rightarrow(iv)$ It follows from Lemma~\ref{type I implies torsion}\par
    $(iv)\Rightarrow (ii)$ It follows from Lemma \ref{type I implies root of unity}. \par
    $(ii)\Rightarrow (v)$ By Remark \ref{correspondence} and Lemma \ref{extVA}, it follows that there exists a virtually abelian group $\widetilde{G}$, such that every irreducible representation of $C^*(G,\sigma)$ has the same dimension as an irreducible representation in $C^*(\widetilde{G})$. But $C^*(\widetilde{G})$ is subhomogeneous by Remark \ref{remarks about C*-algebras of fg VA groups}. It follows that every irreducible representation of $C^*(G,\sigma)$ is uniformly bounded in dimension, and hence $C^*(G,\sigma)$ is subhomogeneous. \par
    $(v)\Rightarrow (vi)$. It is immediate. \par
    $(vi)\Rightarrow (ii)$. It follows from Proposition \ref{fd=>rational}.

We note that $(v)\Rightarrow(iii)$ follows from~\cite{Ho1981} or~\cite{kleppner_type_I_cocycles}.
    \end{proof}

\section{Nuclear dimension of twisted group $C^*$-algebras when the cohomology class is rational.}
\end{theorem}
The main goal of this Section is to compute the nuclear dimension of the twisted group $C^*$-algebra of a finitely generated and virtually abelian group with a rational cohomology class. Because these $C^*$-algebras are subhomogeneous by Theorem \ref{characterization of rational cocycle}, we will be using Theorem \ref{dim_nuc_subhomogeneous}. Hence, we need to describe finite dimensional, irreducible representations. However, Remark~\ref{correspondence} allows us to study the untwisted case. However, we need to be aware of the following:
\begin{enumerate}
    \item The group that we need to study must arise as a central extension of a finite cyclic group by a finitely generated, virtually abelian group. 
    \item In the proof of the untwisted case in \cite{paper_untwisted_case}, an important step was to find $m\in \N$ such that $\dim \Prim_m(C^*(G))=h(G)$. Here, we need to strengthen that, and find $s\in \N$, which might be different from $m$ used in \cite{paper_untwisted_case}, such that \textbf{a particular subset of} $\Prim_{s}(C^*(G))$, defined in Remark~\ref{correspondence}, has covering dimension equal to $h(G)$. This subset will be explicitly defined in (\ref{11}).
\end{enumerate}
We are now ready to set the stage. \par

Let $G$ be a finitely generated, virtually abelian group. Assume that there exists $a\in \mathcal{Z}(G)$ of order $n$.\par

Because $G$ is finitely generated, virtually abelian, it has a normal subgroup, let $N$, that is isomorphic to $\Z^r$ for $r=h(G)$.\footnote{See also Subsection 2.4.}  Because $\Z_n$ is central in $G$, it follows that $\Z^r\oplus\Z_n\cong\Span(N,a)\unlhd G$. \par
Define $L=:C_G(N)$ to be the centralizer of $N$ in $G$. Notice that $\Span(N,a)\subset L$, and that $L$ is the kernel of action $G\curvearrowright N$. Hence $L\unlhd G$. \par
Set $\mt{K}:=[G:L]$ and $F=L/N$. Moreover we fix an $n$-th root of unity $\omega$. We need two results from \cite{paper_untwisted_case}.
\begin{lemma}[Lemma 3.5, \cite{paper_untwisted_case}]\label{minimal orbits}
   Let $\psi\in \TT^r\cong \widehat{N}$. Then $G_{\psi}\geq L$ with equality if and only if the orbit $\mc{O}_{\psi}$ has order $\mt{K}$.  
\end{lemma}
\begin{proposition}[Prop. 3.9, \cite{paper_untwisted_case}]\label{density of minimal stabilizer}
   The set $M:=\{\chi\in \TT^r: G_{\chi}=L\}$ is open and dense in $\TT^r$.
\end{proposition}
We also need the Mackey Machine which describes, as a set, the irreducible representations of $G$.
\begin{theorem}[Thm. 4.28, \cite{KanTay13}\label{thm:Mackey_machine}]
Let $G$ be a discrete group containing a finite index normal abelian group $A$. Let
$\Omega\subseteq\widehat{A}$ be a set-theoretic cross section of orbits\footnote{We create $\Omega$ by taking exactly one element from each orbit (see discussion above \cite[Thm. 4.28]{KanTay13}). There are no other assumptions, so such an $\Omega$ always exists.} under the action $G\curvearrowright \widehat{A}$. Let $\widehat{G}_{\chi}^{(\chi)}$ denote the subset of elements $\sigma\in\widehat{G}_{\chi}$ where there exists $m\in\Z_{>0}$ such that
\begin{equation}\label{eqn:block_diag_restriction}
\sigma\big|_{A}=\chi^{\oplus m}.
\end{equation}
    Then 
    $$\widehat{G}=\{\ind_{G_{\chi}}^G \sigma: \sigma\in \widehat{G}_{\chi}^{(\chi)}, \chi\in \Omega\}.$$
\end{theorem}
Our approach shares some elements with the untwisted case, but one major change is needed. We will illustrate why this is the case.\par
Let 
$$N_{\mt{K}}:=\{\chi \in \widehat{L}_{1D}\,\colon\, G_{\rho(\chi)}=L\}$$
where $\widehat{L}_{1D}$ is the set of 1-dimensional irreducible representations (i.e characters) of $L$, and $\rho(\chi)$ is the restriction of $\chi$ to $N$. By \cite[Prop. 3.10]{paper_untwisted_case} and its proof, $\dim(N_{\mt{K}})=r$. However, in our setting, we need to consider the subset of $N_{\mt{K}}$ with the extra condition that $\chi(a)=\omega$. It turns out that this extra condition can cause the set to become empty, as the following example illustrates. A key step in~\cite{paper_untwisted_case}, is their Lemma 3.2, which allows one to extend a character on $N$ to one on~$L$. However, this is no longer possible with $\Span(N, a)$ replacing $N$ itself, as the following example shows.

\begin{example}
     Let $G=\Z\times\mathbb{H}_3(\Z_2)$ (here $\mathbb{H}_3(\Z_2)$ is the set of upper-triangular $3\times 3$ matrices in $\Z_2$ with all diagonal entries equal to 1). Let $g$ be the nontrivial element of the center of $\mathbb{H}_3(\Z_2)$ and note that in this case $L=G$, thus $L_{ab}\cong \Z\times\Z_2\times\Z_2$, and $g\in[L,L]$, so the nontrivial character on $\langle g\rangle$ does not extend to a 1-dimensional character of $L$. \par
\end{example}
In our previous example, the character extends to a 2-dimensional irreducible representation. This motivated us to consider irreducible representations of $L$ with a higher dimension than 1. \par
Just like \cite{paper_untwisted_case}, our goal is to define a map that is homeomorphic to its image, and its image is a subset of some $\Prim_{s}(C^*(G))$. We first need the following Lemma.
\begin{lemma}\label{embedding lemma}
    Letting $d=|F|$, there is an injective map $\iota:L\to \Z^r\times F$ such that $\iota(N)=d\Z^{r}$.
    \begin{proof}
       Because $F$ is a finite group, $H^2(F;\Z^r)$ is a torsion group (\cite[Corollary III.10.2]{coho}) where all elements have order at most $d$. Thus the map $\Z^r\to \Z^r$ defined via a multiplication with $d$, induces the zero map $H^2(F;\Z^r)\to H^2(F;\Z^r)$. Then by Proposition~\ref{exercise}, we get a commutative diagram as follows

$$\begin{tikzcd}
e\arrow{r}&N\cong\Z^r\arrow{r}\arrow{d}{\bullet d}&L\arrow{r}{q}\arrow{d}{\iota}&F\arrow{r}\arrow{d}{\id}&e\\
e\arrow{r}&\Z^r\arrow{r}&\Z^r\times F\arrow{r}{\widetilde{q}}&F\arrow{r}&e.
\end{tikzcd}$$
Finally, commutativity implies that $\iota(N)=d\Z^r$.
    \end{proof}
\end{lemma}
Let $\{e_1,...,e_r\}$ be the canonical basis of $N\cong \Z^r$ and define $U\subset\widehat{N}\cong\TT^r$ via
$$U=\{\chi\in\widehat{N}:G_{\chi}=L,\chi(e_i)\ne -1\}.$$
Notice that, for each $i=1,2,...,r$, the set $\{\chi\in \widehat{N}:\chi(e_i)\neq -1\}$ is open and dense in $\widehat{N}$. Thus, Lemma \ref{density of minimal stabilizer} implies that $U$ is dense, as an intersection of (finite) open and dense subsets of $\widehat{N}$.
\begin{lemma}\label{existance of irrep on finite group}
  There exists $m\in \Z$ and $\pi\in \Prim_m(C^*(F))$ such that $(\pi \circ q)(a)=\diag(\omega,\ldots,\omega)$.
  \begin{proof}
      Recall that $q$ is the natural map from $L$ to $F$. Start with the character $\pi_0 \colon \langle q(a)\rangle\to \TT$ defined via $\pi_0(q(a))=\omega$. Set $\pi_1:=\ind_{\langle q(a)\rangle}^F \pi_0:L\to \U(|F|/n)$. By construction of the induced representation, and the fact that $a\in \ZZ(L)$, we deduce that $\pi_1(q(a))=\diag(\omega,...,\omega)$. \par

      Notice that $\pi_1$ is a direct sum of irreducible representations. Let $\pi \colon F\to U_m(\C)$ be one of these. The proof is complete because $\pi(q(a))=\diag(\omega,...,\omega)$.
  \end{proof}
\end{lemma}
Recall from complex analysis, that the map $\sqrt[d]{\cdot}\colon \TT/\{-1\}\to \{e^{2\pi i t}: 0\leq t\leq \frac{1}{d}\}$ that sends a complex number to its $d$-th root is continuous. Hence, for each $\chi\in U$ we may define $\sqrt[d]\chi$ to a character defined by $\sqrt[d]{\chi}(e_i)=\sqrt[d]{\chi(e_i)}$. \par
Let $\pi \colon F\to U_m(\C)$ be the irreducible representation defined in Lemma \ref{existance of irrep on finite group}. We define $\varphi_\chi \colon \Z^r\times F\rightarrow U(m)$ via
$$\varphi_\chi(x,y)=\sqrt[d]{\chi}(x)\pi(y)$$
where the multiplication is the scalar multiplication $\C\times M_m(\C)\to M_m(\C)$.\par
Then we define $\Phi \colon U\to \Prim(C^*(G))$ via
$$\Phi(\chi)=[\ind_L^G(\varphi_\chi\circ\iota)].$$
\begin{lemma}\label{Phi is well defined}
The map $\Phi$ defined above is well-defined.
\end{lemma}
\begin{proof}
We need to show that $\ind_L^G(\varphi_\chi\circ\iota)$ is irreducible. 

First, we will show that $\varphi_\chi\circ\iota$ is irreducible. To show this, note for each $g\in F$, there is some $h\in L$ so that $q(h)=g$. By commutativity of the diagram on the proof of Lemma \ref{embedding lemma}, $(\widetilde{q}\circ \iota)(h)=q(h)=g$. So, by construction of the map $\varphi_{\chi}$, it follows that $\varphi_\chi(\iota(h))$ is some scalar times $\pi(g)$. Thus, if $p$ is a projection that commutes with the image of $\varphi_\chi\circ\iota$, then it also commutes with the image of $\pi$. Since $\pi$ is irreducible $p\in\{0,1_m\}$, so $\varphi_\chi\circ\iota$ is irreducible.

Since $\chi\in U$, $G_{\chi}=L$. Moreover, the restriction of $\varphi_\chi\circ\iota$ to $N$ is $\chi^{\oplus m}$. Indeed, let $b\in N$. Then $\varphi_{\chi}\circ \iota(b)=\varphi(\chi)(d\cdot b, e)=\sqrt[d]\chi(d\cdot b)\pi(e)=\sqrt[d]{\chi}(d\cdot b)\cdot 1_m=\left(\sqrt[d]{\chi}(b)\right)^d=\chi(b)\cdot 1_m$. Thus, it follows from the Mackey machine that $\ind_L^G(\varphi_\chi\circ\iota)$ is irreducible.
\end{proof}
For each $k$, define \begin{equation}\label{11}
  A_{k,\omega}^G:=\{[\sigma]\in \Prim_k(C^*(G)): \sigma(a)=\diag(\omega,...,\omega)\} \,. 
\end{equation} \par

\begin{lemma}\label{image of Phi}
    $\Phi(U)\subset A_{\mt{K}\cdot m, \omega}^G$.
    \begin{proof}
        It is clear from the above construction that $\Phi(U)\subset \Prim_{\mt{K}\cdot m}(C^*(G))$. Because $a$ has finite order, it follows that $\iota(a)=(e,b)$ for some $b\in F$ such that $q(a)=b$. Thus, for every $\chi\in U$, we have $\varphi_{\chi}\circ \iota(a)=\pi(b)=(\pi\circ q)(a)=\diag(\omega,\omega,...,\omega)$. Note that the last equality follows from Lemma \ref{existance of irrep on finite group}. 
    \end{proof}
\end{lemma}
Before we can prove that $\Phi$ is continuous, we need to fix the topology on the domain and codomain. Note that $U$ is a subspace of $\TT^r$, and we endow it with the subspace topology. Here we view $\TT^r$ with the usual (metrizable) topology. We will use $d$ to denote a corresponding metric. Recall that $\Prim(C^*(G))$ is endowed with the Jacobson topology.
\begin{lemma}\label{continuity of Phi}
The map $\Phi$ defined above is continuous.
\end{lemma}
\begin{proof}
Let $\chi_n\to \chi$. Then, by construction, $\varphi_{\chi_n}\to \varphi_{\chi}$, where convergence is pointwise on $U(m)$. Hence $\varphi_{\chi_n}\circ \iota \to \varphi_{\chi}\circ \iota$, again pointwise. By \cite[Lemma 4.22]{CW24}, we deduce that $\ind_L^G(\varphi_{\chi_n}\circ\iota)\to \ind_L^G(\varphi_\chi\circ\iota)$ in $\Rep(C^*(G))$. Finally, by  \cite[3.5.8 (p.83)]{Dix77} the convergence also occurs in $\Prim(C^*(G))$. It follows that $\Phi$ is continuous.
\end{proof}
\begin{lemma}\label{intermediate lemma for injectivity}
    Suppose that $\Phi(\chi)=\Phi(\chi')$ for some $\chi,\chi'\in U$. Then $\chi$ and $\chi'$ are in the same orbit under the action $G\curvearrowright \widehat{N}$.
\end{lemma}
\begin{proof}
    Assume that $\Phi(\chi)=\Phi(\chi')$ for some $\chi,\chi'\in U$.
    By the construction of the induced representation, we have
    $$\Phi(\tau)|_{N}\cong\bigoplus_{g\in G/L}(g\cdot\tau)$$ for every $\tau\in U$. In particular, the above formula holds for $\chi$ and $\chi'$. By the proof of \cite[Lemma 4.3]{paper_untwisted_case}, it follows that $\chi$ and $\chi'$ are in the same orbit under the action $G\curvearrowright \widehat{N}$.
\end{proof}
\begin{lemma}\label{injective2}

\begin{enumerate}[label=\roman*.]
    \item For any $\chi\in U$ there exists $\varepsilon>0$ so that $\overline{B_\varepsilon(\chi)}$ is homeomorphic to its image under $\Phi$.
    \item $\dim(A_{\mt{K}\cdot m,\omega}^G)\geq r$. Here $B_\varepsilon(\chi)$ is the ball of radius $\varepsilon$ on our chosen metric.  
\end{enumerate}
\begin{proof}
    {\em i.} Pick any $\chi\in U$. We will find $\varepsilon>0$ such that $\Phi$ is injective on $\overline{B_\varepsilon(\chi)}$. \par
    For $\eta\in U$ define 
$$h(\eta)=\min_{g\in G/L\setminus\{e\}}d(g\cdot\eta,\eta)\,.$$ 
Note that $h$ depends continuously on $\eta$, and is always positive since the stabilizer of $\eta$ is $L$. Pick some $\delta_1>0$ such that $\overline{B_{\delta_1}(\chi)}\subseteq U$. Then by compactness, $h$ must obtain a minimum on $\overline{B_{\delta_1}(\chi)}$; call this minimum $\delta_2$. Let $\varepsilon=\min(\delta_1,\delta_2/3)$. For any distinct $\eta,\eta'\in\overline{B_\varepsilon(\chi)}$, we see that $d(\eta,\eta')<\delta_2\le h(\eta)$ so $\eta'$ cannot be in the $G/L$-orbit of $\eta$. It follows from Lemma~\ref{intermediate lemma for injectivity} that $\Phi(\eta)\ne\Phi(\eta')$. \par
So $\Phi$ is injective on $\overline{B_\varepsilon(\chi)}$. Consider the restriction $$\Phi|_{\overline{B_{\varepsilon}(\chi)}}\colon \overline{B_{\varepsilon}(\chi)}\to \Prim_{\mt{K}\cdot m}(C^*(G))\,.$$
This is an injective and continuous map. Moreover, $\overline{B_{\varepsilon}(\chi)}$ is compact and $\Prim_{\mt{K}\cdot m}(C^*(G))$ is Hausdorff by Remark \ref{remarks about C*-algebras of fg VA groups}. Hence, it is a homeomorphism onto its image.

{\em ii.} By Remark \ref{remarks about C*-algebras of fg VA groups}, $\Prim_{\mt{K}\cdot m}(C^*(G))$ is totally normal. The result follows from \cite[Theorem 6.4]{Pea75}, combined with Lemma \ref{image of Phi}, and the fact that $\dim(\overline{B_{\varepsilon}(\chi)})=r$.
\end{proof}
\end{lemma}
We are now ready to prove our main result.
\begin{theorem}\label{main theorem}
    Let $G$ be a finitely generated, virtually abelian group $G$ and $[\sigma]$ a rational cohomology class. Then $\dimnuc(C^*(G,\sigma))=h(G)$.
    \begin{proof}
     By Remark \ref{correspondence}, there exists a finitely generated, virtually abelian group $\widetilde{G}$, a natural number $n$, and $\omega\in \TT$ such that $\omega^n=1$ and for every $k\in \N$, $\Prim_k(C^*(G),\sigma)\cong A_{k,\omega}^{\widetilde{G}}$. Set $s:=\mt{K}\cdot m$. \par
        By Lemma \ref{injective2}, \cite[Main Theorem]{Win04}, \cite[Prop. 2.14]{BL24}, and the above, we have the following series of inequalities
        $$h(G)\leq \dim(A_{s, \omega}^{\widetilde{G}})\leq \max_k\dim\Prim_k(C^*(G,\sigma))=\dimnuc C^*(G,\sigma)\leq h(G)\,.$$
        Hence, equality must hold everywhere, so the result follows.
    \end{proof}
\end{theorem}

\section{Irrational Twists}
The main goal of this section is to show that, in the case of the higher dimension noncommutative Tori, we never have equality in Theorem \ref{main theorem}. \par
Throughout this section, we will use $\mathcal{T}_r$ to define the set of $r\times r$ skew-symmetric matrices with real coefficients.
\begin{definition}\label{higher dimensional non-commutative torus}
Let $\Theta\in \mathcal{T}_r$. We can define the \emph{$r$-dimensional non-commutative torus $A_{\Theta}$} as the twisted group $C^*$-algebra $C^*(\Z^r,\sigma_{\Theta})$, where $\sigma_{\Theta}$ is the 2-cocycle on $\Z^r$ defined via
$$\sigma_{\Theta}(x,y)=\exp(\pi i \langle\Theta x,y\rangle)$$
for every $x,y\in \Z^r$. \par
Here $\langle \cdot, \cdot\rangle$ is the standard inner product in $\R^r$.
\end{definition}
\begin{remark}\label{all twisted group c* algebras are nc tori}
    It turns out that every twisted group $C^*$-algebra of $\Z^r$ is isomorphic to an $r$-dimensional non-commutative torus. For an explanation why this is true we refer the reader to \cite[Section 2.2]{elliott2006moritaequivalencesmoothnoncommutative}.
\end{remark}
We also need the following definition.
\begin{definition}[Definition 1.6, \cite{simple_nc_torus}] \label{degenerate}
  A matrix $\Theta\in \mathcal{T}_r$ is called \emph{non-degenerate} if whenever $x\in \Z^r$ satisfies $\exp(2\pi i \langle x,\Theta y\rangle)=1$ for every $y\in \Z^r$, then $x=0$. Otherwise, we say that $\Theta$ is \emph{degenerate}.
\end{definition}

\begin{example}\label{rotation algebras skew symmetric matrix}
    Let $\theta\in [0,1]$. Then the rotation algebra $A_{\theta}$ is isomorphic to $A_{\Theta_{\theta}}$, where $\Theta_{\theta}=\begin{pmatrix}0&-\theta\\\theta&0\\ \end{pmatrix}$. Moreover $\Theta_{\theta}$ is non-degenerate if and only if $\theta \notin \Q$. 
\end{example}
Notice that in the above example, the $C^*$-algebra that arises is simple if and only if the cocycle in non-degenerate. This is not a coincidence, and it can actually be generalized by the following result due to N.C. Phillips.
\begin{theorem}[Theorem 3.8, \cite{simple_nc_torus}]\label{when nc torus is simple?}
 Let $\Theta$ be a $r\times r$ skew symmatric real matrix that is non-degenerate, and $r\geq 2$. Then $A_{\Theta}$ is a simple, $A\TT$-algebra.
\end{theorem}
Our goal is to find an upper bound on the nuclear dimension of the $C^*$-algebras $A_{\Theta}$. In order to motivate our method, we start with the following example.
\begin{example}\label{nc torus block digonal}
 Let $$\Theta=\begin{pmatrix}
     0_{k\times k}&0_{k\times \ell}\\ 0_{\ell \times k}& D\\
 \end{pmatrix}$$ where $D\in \mathcal{T}_{\ell}$ is non-degenerate. \par
We have
 $$A_{\Theta}\cong C(\TT^k)\otimes A_{D}$$
 Moreover, by Theorem \ref{when nc torus is simple?}, $A_D$ is a simple, $A\TT$ algebra.
 In particular, if $\ell=2$ and $D$ is the matrix $\Theta_{\theta}$ from Example \ref{rotation algebras skew symmetric matrix} for $\theta \notin \Q$, then
 $$A_{\Theta}\cong C(\TT^k)\otimes A_{\theta}.$$
\end{example}
So, it is natural to search for an upper bound on the nuclear dimension of $C^*$-algebras of the form $C(\TT^k)\otimes A$, where $A$ is a simple, $A\TT$-algebra. It turns out that such an upper bound follows from a result of Tikuisis and Winter (\cite{nuclear_dimension_commutative_tensor_Z}), combined with properties of the nuclear dimension. Therefore, we expect that this is already known to experts. However, we provide a proof for the sake of completion. We also would like to thank Rufus Willett for pointing out the upper bound in \cite{nuclear_dimension_commutative_tensor_Z} to us.
\begin{theorem}\label{upper bound for dim_nuc of commutative tensor classifiable}
  Let $A$ be a simple, separable, unital, infinite-dimensional, $A\TT$-algebra. Then $$\dimnuc(C(\TT^k)\otimes A)\leq \min\{5, k+1\}\,.$$
\end{theorem}
\begin{proof}
    We first prove that $\dimnuc(C(\TT^k)\otimes A)\leq 5$. \par
    By \cite[Corollary 3.1]{z-stability_ASH_simple}, $A$ is $\ZZ$-stable. Hence, it has nuclear dimension 1 by \cite[Theorem A]{dim_nuc_simple} and $$C(\TT^r)\otimes A\cong (C(\TT^k)\otimes \ZZ)\otimes A\,.$$
    But $\dim_{nuc}(C(\TT^k)\otimes \ZZ)\leq 2$ by \cite[Theorem 4.1]{nuclear_dimension_commutative_tensor_Z}.\footnote{The result on \cite{nuclear_dimension_commutative_tensor_Z} is stated with respect to the decomposition rank. However, we the nuclear dimension is always less or equal to the decomposition rank. Hence, the upper bound is also true for the nuclear dimension.} Moreover, \cite[Proposition 2.3(ii)]{WZ10} implies that
    $$\dimnuc(C(\TT^k)\otimes A)\leq (2+1)\cdot (1+1)-1=5\,.$$
    We now show that $\dimnuc(C(\TT^k)\otimes A)\leq k+1$. \par
    Because $A$ is an $A\TT$-algebra, $A=\varinjlim A_n$, where $$A_n=\bigoplus_{i,n=1}^{j_n} M_{d_{i,n}}(C(\TT))\,.$$
    It follows that $C(\TT^k)\otimes A=\varinjlim B_n$, where $$B_n=\bigoplus_{i,n=1}^{j_n} M_{d_{i,n}}(C(\TT))\otimes C(\TT^k)\cong \bigoplus_{i,n=1}^{j_n} M_{d_{i,n}}(C(\TT^{k+1}))\,.$$
    Notice that $\dimnuc(A_n)=k+1$ for every $n\in \N$. Hence, \cite[Proposition 2.3(iii)]{WZ10} implies
    $$\dimnuc(A)\leq \liminf \dimnuc(A_n)=k+1\,.$$
    Combining our two upper bounds, we deduce that
    $$\dimnuc(C(\TT^k)\otimes A)\leq \min\{5, k+1\}\,.$$
    
\end{proof}
We also need the following Remark.
\begin{remark}\label{dim_nuc_morita_equivalent}
    Strongly Morita equivalent separable $C^*$-algebras have the same nuclear dimension. Indeed, let $A$ and $B$ be strongly Morita equivalent separable $C^*$-algebras. Then $A\otimes \mathcal{K}\cong B\otimes \mathcal{K}$ (\cite{morita_equivalence_stable_isom}). Hence
    $$\dimnuc(A)=\dimnuc(A\otimes \mathcal{K})=\dimnuc(B\otimes \mathcal{K})=\dimnuc(B)\,.$$
    The first and third equality follow from \cite[Corollary 2.8(i)]{WZ10}.
\end{remark}
Before we can prove our main result of the section, we need to give some background information. First, note that $\SO(n,n|\Z)$ is a certain subset of $\SO(n,\Z)$. For the exact definition we refer the reader to \cite[p.8]{elliott2006moritaequivalencesmoothnoncommutative}. Moreover, there exists a densely defined action $\SO(n,n|\Z)\curvearrowright \mathcal{T}_n$. This action is described on \cite[p.8]{elliott2006moritaequivalencesmoothnoncommutative}, and it was first defined in \cite{action_of_skew_symmetric_matrices}. We also need the following results from the literature.
\begin{proposition}[Proposition 3.3, \cite{elliott2006moritaequivalencesmoothnoncommutative}]\label{equivalent to nice matrix}
Let $\Theta\in \mathcal{T}_r$. Then there exists $g\in \SO(r,r|\Z)$ such that $$g\cdot \Theta=\begin{pmatrix}
            0&0\\0&D
        \end{pmatrix},$$
        where $D\in \mathcal{T}_d$ for $1\leq d \leq r$ and is non-degenerate.
    
\end{proposition}
\begin{proposition}[Theorem 1.1, \cite{Morita_equivalence_nc_tori}]\label{criterion_for_morita_equivalence}
Let $\Theta\in \mathcal{T}_r$ and $g\in \SO(r,r|\Z)$. If $g\cdot \Theta$ is defined, then $A_{\Theta}$ and $A_{g\cdot \Theta}$ are strongly Morita equivalent. 
\end{proposition}
Now we are ready to prove our main result.
\begin{theorem}\label{irrational case on Z^r}
    Let $\Theta$ be a skew symmetric real $r\times r$ matrix. Assume that the cohomology class $[\sigma_{\Theta}]$  of the associated cocycle is irrational. Then $$\dimnuc(A_{\Theta})\leq \min\{5,r-1\}\,.$$
    In particular, $\dimnuc(A_{\Theta})<r$.
    \begin{proof}
        By Proposition \ref{equivalent to nice matrix}, there exists a non-degenerate skew symmetric, real $d\times d$ matrix $D$, where $1\leq d\leq r$, and $g\in \SO(r,r|\Z)$ such that
        $g\cdot \Theta=\begin{pmatrix}
            0&0\\0&D
        \end{pmatrix}$
        where the multiplication is the action $\SO(r,r|\Z)\curvearrowright \mathcal{T}_r$. \par
        By assumption, $[\sigma_{\Theta}]$ is irrational. This implies that $d\geq 2$.
        By Proposition \ref{criterion_for_morita_equivalence}, $A_{\Theta}$ and $A_{g\cdot \Theta}$ are strongly Morita equivalent. Hence, Example \ref{nc torus block digonal}, implies that
        $A_{\Theta}$ is strongly Morita equivalent to $C(\TT^{r-d})\otimes A_D$. Moreover, Theorem \ref{when nc torus is simple?} implies that $A_D$ is a simple, $A\TT$-algebra. Hence
        $$\dimnuc(A_{\Theta})=\dimnuc(C(\TT^{r-d})\otimes D)\leq \min\{5,r-d+1\}\leq \min\{5,r-1\}\,.$$
        Notice that the first equality is due to Remark \ref{dim_nuc_morita_equivalent}, the first inequality follows from Theorem \ref{upper bound for dim_nuc of commutative tensor classifiable}, while the second holds because $d\geq 2$. Proof is complete.
    \end{proof}
\end{theorem}
\begin{corollary}\label{equivalene of rational and dimnuc=rank}
Let $\sigma$ be a 2-cocycle on $\Z^r$. The following are equivalent.
\begin{enumerate}[label=\roman*.]
    \item $\dimnuc(C^*(\Z^r,\sigma))=r$
    \item $[\sigma]$ is rational.
\end{enumerate}
\begin{proof}
 It follows from Remark \ref{all twisted group c* algebras are nc tori}, Theorem \ref{main theorem} and Theorem \ref{irrational case on Z^r}.  
\end{proof}
\end{corollary}

\bibliographystyle{abbrv}
\small

\bibliography{main}

\end{document}